\documentclass{article}

\usepackage[left=1in, right=1in, top=1in, bottom=1in, noheadfoot]{geometry}
\usepackage{amsmath}
\usepackage{amssymb}

\newcommand{\ignore}[1]{}

\newtheorem{theorem}{Theorem}

\newtheorem{lemma}{Lemma}
\newtheorem{lemma_repeat}{Repeat of Lemma}
\newtheorem{theorem_repeat}{Repeat of Theorem}

\newenvironment{proof}{\emph{Proof. } }{\hfill $\Box$}

\title{Error Bounds on Derivatives during Simulations}
\author{Gregory Bard\thanks{Dept. of Math., Stat., and Comp. Sci., Jarvis Hall Science Wing, University of Wisconsin---Stout, Menomonie, Wi, 54751.} \ and Alexander Basyrov}

\begin{document}
\maketitle

\pagestyle{empty}
\thispagestyle{empty}

\vspace{-0.2in}
\begin{abstract}
The methods commonly used for numerical differentiation, such as the ``center-difference formula''
and ``four-points formula'' are unusable in simulations or real-time data analysis because they
require knowledge of the future. In Bard'11, an algorithm was shown that generates formulas that
require knowledge only of the past and present values of $f(t)$ to estimate $f'(t)$. Furthermore,
the algorithm can handle irregularly spaced data and higher-order derivatives. That work did not
include a rigorous proof of correctness nor the error bounds. In this paper, the correctness and
error bounds of that algorithm are proven, explicit forms are given for the coefficients,
and several interesting corollaries are proven.
\end{abstract}


\emph{MSC2010 Subject Classification:} 
65D25, 65D15, 68U20, 68W30, 68W40.

\section{Introduction}
In numerical analysis, one is often taught that for positive $h$, rather than use
$$ f'(t) \approx \frac{f(t) - f(t-h)}{h} \hspace{0.5in}\mbox{(the ``backward difference'' formula)}$$
it is better to use
$$ f'(t) \approx \frac{f(t+h) - f(t-h)}{2h} \hspace{0.5in}\mbox{(the ``center difference'' formula)}$$
or better still
$$ f'(t) \approx \frac{4}{3} \left [ \frac{f(t+h) - f(t-h)}{2h} \right ] - \frac{1}{3} \left [ \frac{f(t+2h) - f(t-2h)}{4h} \right ] \hspace{0.5in}\mbox{(the ``four points'' formula)}$$
because of the error of the approximations.
The error in the first case is $O(h)$, compared to $O(h^2)$ in the second case and $O(h^4)$, where
$O(h^m)$ indicates some function that is bounded when divided by $h^m$ for all sufficiently small values of the step $h$.

In simulations or in real-time data analysis, these more advanced formulas are of no use,
because they require knowledge of the future---i.e. $f(t+h)$ and $f(t+2h)$. However, the future
is either unknown (in the case of real-time data analyses) or not yet calculated (in the case of simulations).

In a previous paper [1], Bard showed how to compute similar formulas such as
$$ \frac{1}{4} f\left ( t - 4h \right ) -
\frac{4}{3} f\left ( t - 3h \right ) +
3 f\left ( t - 2h \right ) -
4 f\left ( t - h \right ) +
\frac{25}{12} f\left ( t \right ) =
f'\left (t \right )
- \frac{1}{5} h^4 f^{(5)}\left (t \right )
+ \frac{1}{3} h^5 f^{(6)}\left (t \right )
+ \cdots
 $$
that work with only the
knowledge of the past and present, but that do not require knowledge of the
future. Furthermore, the algorithm presented there
also allows for irregularly spaced sampling of $f(t)$, and arbitrary numbers of data points,
as well as second and higher derivatives. The error bounds of the formulas from that paper were handled
heuristically, and only a sketch of the proof of correctness was given. Here we give a rigorous
proof, explicit forms of the coefficients of those formulas, and several rigorous statements about the error bounds. The obtained formulas could be used to reduce the running time while keeping the accuracy of derivative estimates.

For a survey of previous work in this topic, see Section~IX of [1]; that paper also contains several extended examples,
and SAGE code for carrying out the algorithm. This document can be thought of as a sequel to that paper.

\ignore{
\section{The Algorithm}

\begin{enumerate}
\item [Input:] Any $n$ distinct real numbers $\delta_1, \delta_2, \ldots, \delta_n$, and an integer $k$ such that $0<k<n$.
\item [Output:] A formula for the $k$th derivative of any ?-times differentiable $f(t)$, in terms of $f(t + \delta_j h)$,
for any positive real $h$, with error proportional to $h^{n-k}$ and the $n$th derivative of $f(t)$.
\item Define the $n\times n$ matrix $A$ such that $A_{ij} = \delta_j^{i-1}$.
\item Define the $n$-dimensional vector $\vec{b}$ to be all zeros, except the $b_{k+1} = k!$.
\item Let the $n$-dimensional vector $\vec{c}$ be the solution to $A\vec{c}=\vec{b}$.
\item [Note:] The matrix $A$ is always non-singular so $\vec{c}$ will always exist and be unique.
\item Return the formula
$$ \frac{1}{h^k} \sum_{j=1}^{j=n} c_j f( t + \delta_j h ) \approx f^{(k)}( t ) $$
\end{enumerate}

We will now prove that
$$ \frac{1}{h^k} \sum_{j=1}^{j=n} c_j f( t + \delta_j h ) = f^{(k)}( t ) + \sum_{j=1}^{j=n} \sum_{i=n}^{\infty} c_j f^{(i)}( t ) \frac{\delta_j^i h^{i-k}}{i!} $$
}

\section{Initial Calculations}

Assuming that $f$ is at least $(n-1)$ times differentiable, we use the following form of Taylor's Theorem
$$ f( t + h ) = \left( \sum_{i=0}^{n-1} f^{(i)}(t) \frac{h^i}{i!} \right) + R_n$$
where $f^{(i)}$ represents the $i$th derivative of $f(t)$, and $R_n$ is the remainder term in one of the well-known forms.
Two of those forms will be explored in detail during Section~4.
We define the 0th derivative of a function to be the function itself.

Let us suppose that $\delta_1, \delta_2, \ldots, \delta_n$ are distinct real numbers. Using the formula above we arrive at the following representations for $f(t+\delta_j h)$
\begin{equation}
    \label{eq:taylor_for_f_with_deltas}
  f(t + \delta_j h) =
  \left(\sum_{i=0}^{n-1} f^{(i)}(t) \frac{\delta^i_j h^i}{i!} \right) + R_{n,j}
\end{equation}

For any integer $k$ such that $0<k<n$, any real non-zero $h$ we are looking for a linear combination of terms $f(t + \delta_j h)$ that approximates $f^{(k)}(t)$ term up to some reasonable remainder term. More precisely, we start with the linear combination
$ c_1 f(t + \delta_1 h) + \cdots + c_n f(t+ \delta_n h)$,
rewrite it using \eqref{eq:taylor_for_f_with_deltas}, and look for the values of $c_j$ that eliminate as many terms as possible except for the terms involving $f^{(k)}(t)$.
\begin{eqnarray}
\sum_{j=1}^{n} c_j f( t + \delta_j h ) & = &
\sum_{j=1}^{n} \left [ c_j \left(\sum_{i=0}^{n-1} f^{(i)}( t ) \frac{\delta_j^i h^i}{i!} \right) + c_j R_{n,j}\right ]\\
\label{eq:regrouped_version}
& = &
\left[ \sum_{i=0}^{n-1} f^{(i)}(t) \frac{h^i}{i!}
\left( \sum_{j=1}^{n} c_j \delta_j^i \right) \right] + \sum_{j=1}^n c_j R_{n,j}
\end{eqnarray}

It is useful to note that in the above
we interchanged two summation symbols, which was legal because both sums are finite;
we also pulled out of the summation over $j$ any factors that do not depend on $j$. So far, everything we have said is true for all real $c_j$s.

If the numbers $c_1, \ldots, c_n$ are chosen so that the following conditions are satisfied
\begin{eqnarray}
\label{eq:equations_vanish}
\sum_{j=1}^n c_j \delta^i_j & = & 0 \qquad \text{for} \  0 \le i \le n-1,  i \not= k \\ 
\label{eq:equation_k_factorial}
\sum_{j=1}^n c_j \delta^k_j & = & k! 
\end{eqnarray}
then \eqref{eq:regrouped_version} dramatically simplifies to become
\begin{equation}
  \sum_{j=1}^{n} c_j f( t + \delta_j h ) = f^{(k)}(t) h^k + \sum_{j=1}^n c_j R_{n,j}
\end{equation}
which yields the desired result
\begin{equation}
  \label{eq:final_result_for_algorithm}
  \frac{1}{h^k} \sum_{j=1}^{n} c_j f( t + \delta_j h ) = f^{(k)}(t) + \frac{1}{h^k} \sum_{j=1}^n c_j R_{n,j}
\end{equation}

Accordingly, to simplify matters, we set
\begin{equation}
  \label{eq:basic_form_of_error_term}
  R = - \frac{1}{h^k} \sum_{j=1}^n c_j R_{n,j}
  \hspace{0.25in} \mbox{ therefore } \hspace{0.25in}
  \frac{1}{h^k} \sum_{j=1}^{n} c_j f( t + \delta_j h ) = f^{(k)}(t) + R
\end{equation}
Thus $R$ represents the error of the formula, and we will seek to place bounds on $R$ in Section~\ref{bounds:sec}.

\subsection{Existence and Uniqueness}

We still have to prove the existence of
numbers $c_1, \ldots, c_n$ that satisfy \eqref{eq:equations_vanish} and \eqref{eq:equation_k_factorial}. In fact, we will prove that such numbers exist, and
that they are unique.
The equations represented by \eqref{eq:equations_vanish} and \eqref{eq:equation_k_factorial} could be rewritten as a system of linear equations $A \vec c = \vec b$ with $A_{ij} = \delta_j^{i-1}$ and
$b_{k+1} = k!$ with the rest of coordinates of $\vec b$ are zero.
Matrix $A$ is an example of a Vandermonde matrix; $A$ is invertible precisely if and only if all of the $\delta_1, \ldots, \delta_n$ are
distinct (see Equation~\eqref{basic_vandermonde_determinant}).
Because matrix $A$ is invertible, there exists unique solution $\vec c = A^{-1} \vec b$, which is exactly what is required. In fact, vector $\vec c$ is a column of matrix $A^{-1}$ multiplied by $k!$.

The above calculations could be stated as the following algorithm.

\subsection{The Algorithm}

\begin{enumerate}
\item [Input:] Any $n$ distinct real numbers $\delta_1, \delta_2, \ldots, \delta_n$, and an integer $k$ such that $0<k<n$.
\item [Output:] A formula for the $k$th derivative of any $n$-times differentiable $f(t)$, in terms of $f(t + \delta_j h)$,
for any positive real $h$, with error proportional to $h^{n-k}$ and the $n$th derivative of $f(t)$.
\item Define the $n\times n$ matrix $A$ such that $A_{ij} = \delta_j^{i-1}$.
\item Define the $n$-dimensional vector $\vec{b}$ to be all zeros, except the $b_{k+1} = k!$.
\item Let the $n$-dimensional vector $\vec{c}$ be the solution to $A\vec{c}=\vec{b}$.
\item Return the formula
$$ \frac{1}{h^k} \sum_{j=1}^{j=n} c_j f( t + \delta_j h ) \approx f^{(k)}( t ) $$
\end{enumerate}

\subsection{The Condition Number of $A$}


We would like to note that
because $A$ is a Vandermonde matrix, some readers maybe concerned about the numerical stability of solving $A\vec{c}=\vec{b}$.
This is because the condition number of a Vandermonde matrix can become extremely tiny as any two data points move
close together. However, this is not an issue, because in any practical situation numbers $\delta_j$ would be integers that do not depend on the value of the step $h$.
Furthermore, the value of $n$ would satisfy $n<20$. Moreover, using computer algebra packages,
one can solve $A\vec{c}=\vec{b}$ using exact rational arithmetic in less than a second, totally avoiding any floating-point computations of any kind.
Therefore, the condition number of $A$ is immaterial, and $\vec{c}$ will be known exactly.


%


\section{Explicit Formulas for the $c_i$s}

While the algorithm will very quickly produce the correct values of $\vec{c}$, it might be useful
to have some explicit formulas in order to deduce properties of the $c_i$s. We will accomplish this
by first finding some formulas for the determinants of the minors of a Vandermonde, and then
use those with Cramer's rule to obtain the explicit formulas for the $c_i$s.

\subsection{Some Useful Notation}
We define a particular $n \times n$-matrix $V(y_1, y_2, \ldots, y_n)$ via
$V_{ij} = y_{j}^{i-1}$.
\ignore{Some textbooks call $V$ the Vandermonde Matrix, while other books call
$V^T$ the Vandermonde matrix. }

Recall that $A=V(\delta_1, \delta_2, \ldots, \delta_n)$ in our
algorithm. We will abbreviate this $A=V(\vec{\delta})$. Furthermore, we will use a hat to
indicate the removal of one of the deltas. Explicitly,
$$ \widehat{\delta_j} = (\delta_1, \ldots, \delta_{j-1}, \delta_{j+1}, \ldots, \delta_n) $$

Then it shall be useful to denote by $\Delta$ the largest absolute value of any $\delta$.
Let $\epsilon_j$ be the distance from $\delta_j$ to the nearest other $\delta$. Finally,
let $\epsilon$ signify the smallest of all the $\epsilon_j$s. Explicitly,
\begin{equation}
\Delta = \max \{ |\delta_1|, \ldots, |\delta_n| \}
\quad\mbox{ and }\quad
    \epsilon_j = \min_{i \not= j} | \delta_j - \delta_i |
\quad\mbox{ with }\quad
    \epsilon = \min \{ \epsilon_1, \ldots, \epsilon_n \}
\label{eq:definition_of_stuff}
\end{equation}

Since we assumed that $h>0$, and $\delta_1, \ldots, \delta_n$ are real numbers, we define the \emph{interval of consideration} $I$ to be the smallest closed interval containing all of $(t + \delta_i h)$ values:
\begin{equation}
\label{eq:definition_of_I}
  I = \{x:\  \min_{i=1,\ldots, n} ( t + \delta_i h) \le x \le \max_{i=1,\ldots, n} (t + \delta_i h) \}
\end{equation}

\subsection{The Determinants of the Minors of a Vandermonde Matrix}
In any case, the determinant of $V$, which we denote $v$,
has a well-known formula
\begin{equation}\label{basic_vandermonde_determinant}
v(y_1, y_2, \ldots, y_n) = \prod_{1 \le i < j \le n} \left ( y_j - y_i \right )
\end{equation}
which is simply the product of the differences of all possible distinct pairs of $y_i$ and $y_j$,
taking care to always subtract the $y$ of higher index from the $y$ of lower index.

The following lemma establishes a formula for the determinant of a minor of a Vandermonde matrix. We presume that this must have been known for quite some time, but we could not find a proof
of it anywhere, and we find the following proof both short and simple.

\begin{lemma}
Consider $V(y_1, y_2, \ldots, y_n)$, with all $y$s distinct.
Let $M(i, j)$ denote the minor formed by
deleting the $i$th row and $j$th column from $V$.
The determinant of $M(i,j)$ is given by
\begin{equation}
m(i,j) = \det M(i,j) =
v( \widehat{y_j} ) \sigma_{n-1,n-i}(\widehat y_j)
\end{equation}
where $\sigma_{n-1,n-i}$ is the $(n-i)$-th-degree symmetric polynomial in $n-1$ variables
\end{lemma}

The proofs of the lemma can be found in Appendix~\ref{lemma_1_proof_appendix}, of the full version of the paper on {\tt arXiv.org}.






\subsection{Cramer's Rule}
Now suppose we want an explicit formula for $c_j$ in $\vec{c}$, as produced by our algorithm.
As in our algorithm and Lemma~1, we define $A = V(\vec{\delta})$.
Using Cramer's Rule, we can take $A$ but replace column $j$ with $\vec{b}$; the determinant
of that modified matrix, divided by the determinant of the original $A$, equals the value of $c_j$.
We should expand the determinant of the modified matrix on column $j$, which is equal to
$\vec{b}$, because $\vec{b}$ is zero in all but one entry. We would then obtain:

\begin{eqnarray}
\label{eq:formula_for_cj_raw}
c_j & = &
\frac{ (-1)^{1+j} A_{1j} \det M_{1j} + (-1)^{2+j} A_{2j} \det M_{2j} + \cdots + (-1)^{n+j} A_{nj} \det M_{nj} }{\det A } \nonumber \\
 & = & (-1)^{k+j+1} k! \frac{ v( \widehat \delta_j )
\sigma_{n-1,n-k-1}(\widehat \delta_j ) }
{ v( \vec{\delta} ) }
\end{eqnarray}
\ignore{where $\widehat \delta_j = (\delta_1, \ldots, \delta_{j-1}, \delta_{j+1}, \ldots, \delta_n)$ and $\delta = (\delta_1, \ldots, \delta_n)$.}

However, we should observe that
$$ \frac{ v( \widehat \delta_j ) }
{ v( \vec{ \delta } ) } = \frac{1}{
\left ( \prod_{1 \le i < j} \left ( \delta_j - \delta_i \right ) \right )
\left ( \prod_{j < i \le n} \left ( \delta_i - \delta_j \right ) \right ) } $$
because all terms in the denominator not involving $\delta_j$ will also be found in the numerator.
Therefore, we can conclude
$$ c_j = (-1)^{k+j+1} \frac{k!\sigma_{n-1,n-k-1}( \widehat \delta_j  ) }
{\left ( \prod_{1 \le i < j} \left ( \delta_j - \delta_i \right ) \right )
\left ( \prod_{j < i \le n} \left ( \delta_i - \delta_j \right ) \right ) } $$
or simply
\ignore{$$ c_j = (-1)^{k+j+1} \frac{k!\sigma_{n-1,n-k-1}( \widehat \delta_j  ) }
{ (-1)^{n-j-1} \prod_{i\not=j} \left ( \delta_j - \delta_i \right ) } $$}
$$ c_j = \frac{(-1)^{n-k} k!\sigma_{n-1,n-k-1}( \widehat \delta_j  ) }
{\prod_{i\not=j} \left ( \delta_j - \delta_i \right ) } $$

When we specialize the above formula to the case $\delta_i= -i$, for $k=1$ we obtain
\[  c_j = (-1)^j \binom{n}{j} \left( - \frac{1}{j} + 1 + \frac{1}{2} + \cdots + \frac{1}{n}\right),
\]
and for $k=2$ we obtain
\[   c_j = (-1)^j 2 \binom{n}{j} \left( \sigma_{n,2}\left(1, \frac12, \ldots, \frac1n\right) - \frac1j \left(- \frac{1}{j} + 1 + \frac{1}{2} + \cdots + \frac{1}{n}\right) \right)
\]
The proof of the special cases can be found in Appendix \ref{appendix_special_cases}, of the full version of the paper on {\tt arXiv.org}.


\ignore{
\subsubsection{Lower Degree Versions}

If $n-k-1$ (for Alex's Version) or $n-j$ (for Greg's Version) is too high a degree for convenience, one
can substitute a lower-degree version using the following identity:
$$ \frac{\sigma_{n, d}(y_1, y_2, y_3, \ldots, y_n) }{ \sigma_{n, n}(y_1, y_2, y_3, \ldots, y_n) } =
\frac{\sigma_{n, d}(y_1, y_2, y_3, \ldots, y_n) }{ (y_1)(y_2)(y_3)\cdots(y_n) } =
\sigma_{n, n-d}(y_1^{-1}, y_2^{-1}, y_3^{-1}, \ldots, y_n^{-1})
$$
provided that all the $y_i$s are non-zero---which might not necessarily be true. For example, if
one uses the present plus 4 data points from the past, one of the deltas is zero.
}


\section{Error Bound Theorems}\label{bounds:sec}

In this section we return to the formula used in the main algorithm, and make the error terms more explicit, and provide some useful error term estimates.

\ignore{
\subsection{Introduction}

As a reminder, we started with the Taylor expansions \eqref{eq:taylor_for_f_with_deltas},
\begin{equation}
    \label{eq:taylor_for_f_with_deltas_repeat}
  f(t + \delta_j h) =
  \left(\sum_{i=0}^{n-1} f^{(i)}(t) \frac{\delta^i_j h^i}{i!} \right) + R_{n,j}
\end{equation}
which were arranged into a linear combination \eqref{eq:final_result_for_algorithm}
\begin{equation}
    \frac{1}{h^k} \sum_{j=1}^{n} c_j f( t + \delta_j h ) = f^{(k)}(t) + \frac{1}{h^k} \sum_{j=1}^n c_j R_{n,j}
\end{equation}
so that we have
\begin{equation}
  f^{(k)}(t) = \frac{1}{h^k} \sum_{j=1}^{n} c_j f( t + \delta_j h ) - \frac{1}{h^k} \sum_{j=1}^n c_j R_{n,j}
\end{equation}
and we set
\begin{equation}
  \label{eq:basic_form_of_error_term}
  R = - \frac{1}{h^k} \sum_{j=1}^n c_j R_{n,j}
\end{equation}
so that
\begin{equation}
  f^{(k)}(t) = \frac{1}{h^k} \sum_{j=1}^{n} c_j f( t + \delta_j h ) + R
\end{equation}
}

\subsection{A Simpler Form of a Particular Sum}

We are about to use the sum $S = \sum_{j=1}^n c_j \delta_j^n $ in the following discussion, so we start by working on a simpler closed form of the sum.

\begin{lemma}
    \label{lemma:for_cool_sum}
  The sum $S$ has the following closed form:
   $ S = \sum_{j=1}^n c_j \delta_j^n = (-1)^{k+n+1} k! \sigma_{n,n-k}(\vec{\delta}) $
\end{lemma}

We will need the following estimate on $S$ as well
\begin{lemma}
  \label{lemma:cool_sum_absolute}
  The sum $S = \sum_{j=1}^n c_j \delta_j^n $ satisfies
    $ \displaystyle |S| \le \sum_{j=1}^n | c_j \delta_j^n | \le \frac{ \Delta^{2n-k-1} n!  }{\epsilon^{n-1} (n-k-1)! } $
\end{lemma}

The proofs of the lemmas can be found in Appendices \ref{lemma_2_proof_appendix} and \ref{lemma_3_proof_appendix}, of the full version of the paper on {\tt arXiv.org}.

\subsection{Error Term Estimates}
The easiest to remember form of the error terms in \eqref{eq:taylor_for_f_with_deltas} is the Lagrange form of the error term:
\begin{equation}
  R_{n,j} = \frac{f^{(n)}(\xi_j) \delta_j^n h^n}{n!}
\end{equation}
for some $\xi_j$ in the interval between $t$ and $(t+\delta_j h)$. Using equation \eqref{eq:basic_form_of_error_term} we immediately have
\begin{equation}
  \label{eq:full_error_term_lagrange}
  R = - \frac{1}{h^k} \sum_{j=1}^n c_j \frac{f^{(n)}(\xi_j) \delta_j^n h^n}{n!} = - \frac{h^{n-k}}{n!} \sum_{j=1}^n c_j \delta_j^{n}  f^{(n)}(\xi_j).
\end{equation}
While this form of the error term is correct and exact, it is hard to use in practice, because the $\xi$s are unknown.

\begin{theorem}
If
\begin{enumerate}
  \item $f$ is $n$-times differentiable on an open set containing interval of consideration $I$,
  \item $| f^{(n)} (\tau) | \le M$ for all $\tau \in I$, with $I$ defined in \eqref{eq:definition_of_I}
  \item $\Delta$, $\epsilon_j$, $\epsilon$ as defined in \eqref{eq:definition_of_stuff}
\end{enumerate}
  then the algorithm's error term (see Equation~\eqref{eq:basic_form_of_error_term}) satisfies the following estimate
  \begin{equation}
    \label{eq:first_esitmate}
    |R| \le \frac{  M \Delta^{2n-k-1} }{\epsilon^{n-1} (n-k-1)!} h^{n-k}
  \end{equation}
\end{theorem}
\begin{proof}
This is a direct calculation starting with \eqref{eq:full_error_term_lagrange}, and using Lemma \ref{lemma:cool_sum_absolute}:
\begin{equation}
  |R| \le
  \frac{h^{n-k}}{n!} \sum_{j=1}^n | c_j \delta_j^n f^{(n)}(\xi_j) | \le
  M \frac{h^{n-k}}{n!} \sum_{j=1}^n | c_j \delta_j^n | \le
  M \frac{h^{n-k}}{n!} \frac{ \Delta^{2n-k-1} n!  }{\epsilon^{n-1} (n-k-1)! }
\end{equation}
Rearranging terms we arrive at the estimate we're looking for.
\end{proof}

We note that another easy to remember form of error term in Taylor theorem is the \emph{little-oh} form. In our context, we rewrite \eqref{eq:taylor_for_f_with_deltas} in the form
\begin{equation}
\label{eq:individual_estimate_little_o_form_new}
  R_{n,j} =  \frac{f^{(n)}(t) \delta_j^n h^n}{n!} + o(\delta_j^n h^n)
\end{equation}
We must confess that in the above we wrote an extra term in the Taylor expansion for future use plus the actual error term. Since only $h$ is typically considered as a parameter approaching zero, and $\delta_j$ is a constant we can replace $o(h^n\delta_j^n)$ with $o(h)$, and
using that, we write \eqref{eq:basic_form_of_error_term} as
\begin{multline}
  \label{eq:error_estimate_little_o_form}
  R = - \frac{1}{h^k} \sum_{j=1}^n c_j \left( \frac{f^{(n)}(t) \delta_j^n h^n}{n!} + o(h^n) \right)  =
  - \frac{h^{n-k} f^{(n)}(t) }{n!}  \left( \sum_{j=1}^n c_j \delta_j^n \right) + o(h^{n-k}) \\
  = \frac{(-1)^{n-k} h^{n-k} k!  f^{(n)}(t) \sigma_{n,n-k}(\vec{\delta}) }{n!} + o(h^{n-k})
\end{multline}
where we used Lemma \ref{lemma:for_cool_sum} and the $o(h^{n-k})$ term is the sum of $n$ different $o(h^{n-k})$ terms
coming from \eqref{eq:individual_estimate_little_o_form_new}.
We state this result as the following theorem.

\begin{theorem}
\label{theorem:little_o_error_term}
If $f$ is $n$-times differentiable on an open set containing interval of consideration $I$
then the algorithm's error term in \eqref{eq:basic_form_of_error_term} is
\begin{equation}
    \label{eq:R_in_little_o_form}
  R = \frac{(-1)^{n-k} h^{n-k} k!  f^{(n)}(t) \sigma_{n,n-k}(\vec{\delta}) }{n!} + o(h^{n-k})
\end{equation}
\end{theorem}

Sometimes it is useful to know if the algorithm's formula will always underestimate or always overestimate
the value of $f^{(k)}(t)$. Using
Theorem \ref{theorem:little_o_error_term} we are able to prove the following additional result.
\begin{theorem}
  Assume that all $\delta_j < 0$
  \begin{enumerate}
    \item If $f^{(n)}(t) > 0$, then the error term $R$ is positive for all sufficiently small values of $h$, and the estimate of $f^{(k)}(t)$ provided by \eqref{eq:final_result_for_algorithm} is an underestimate.
    \item If $f^{(n)}(t) < 0$, then the error term $R$ is negative for all sufficiently small values of $h$ and the estimate of $f^{(k)}(t)$ provided by \eqref{eq:final_result_for_algorithm} is an overestimate. 
    \end{enumerate}
\end{theorem}
The proof of the theorem could be found in Appendix \ref{appendix:under_over_estimates}, of the full version of the paper on {\tt arXiv.org}.

\section{Acknowledgements}
The authors would like to acknowledge Dr. Mingshen Wu, Dr. Keith Wojciechowski, the referees of Modeling, Simulation, and Visualization 2011, and Joseph Bertino, recently graduated student of Fordham, who all commented on previous drafts of this paper.

\section*{The Bibliography}

\begin{enumerate}
\item G. Bard. ``Numerically Estimating Derivatives during Simulations.''
\emph{Proceedings of the 2011 International Conference on Modeling, Simulation \& Visualization Methods (MSV'11).}
(H. Arabnia and L. Deligiannidis, Eds.)
CSREA Press, 2011. (pp. 341--347).
ISBN: 1-60132-192-9
\end{enumerate}

\appendix

\section{Proof of Lemma 1}\label{lemma_1_proof_appendix}

\begin{lemma_repeat}
Consider $V(y_1, y_2, \ldots, y_n)$, with all $y$s distinct.
Let $M(i, j)$ denote the minor formed by
deleting the $i$th row and $j$th column from $V$.
The determinant of $M(i,j)$ is given by
\begin{equation}
m(i,j) = \det M(i,j) =
v( \widehat{y_j} ) \sigma_{n-1,n-i}(\widehat y_j)
\end{equation}
where $\sigma_{n-1,n-i}$ is the $(n-i)$-th-degree symmetric polynomial in $n-1$ variables
\end{lemma_repeat}

\begin{proof}
Instead of focusing on the single formula for the value of $m(i,j)$ we will find all values of $m(1, j), \ldots, m(n,j)$ in one calculation.
We start by setting up the matrix $V(y_1, \ldots, y_{j-1}, x, y_{j+1}, \ldots, y_n)$
\begin{equation}
  V(y_1, \ldots, y_{j-1}, x, y_{j+1}, \ldots, y_n)
  =
  \begin{bmatrix}
    1 & \ldots & 1 & 1 & 1 & \ldots & 1 \\
    y_1 & \ldots & y_{j-1} & x & y_{j+1} & \ldots & y_n \\
    y^2_1 & \ldots & y^2_{j-1} & x^2 & y^2_{j+1} & \ldots & y^2_n \\
    \vdots &  & \vdots & \vdots & \vdots & & \vdots \\
    y^{n-1}_1 & \ldots & y^{n-1}_{j-1} & x^{n-1} & y^{n-1}_{j+1} & \ldots & y^{n-1}_n     \\
  \end{bmatrix}
\end{equation}
and observing that the determinant of the matrix could be obtained via expansion along the $j$th column to yield an $(n-1)$th degree polynomial in $x$:
\begin{multline}
  p_{n-1}(x) = v(y_1, \ldots, y_{j-1}, x, y_{j+1}, \ldots, y_n)
  =
  a_0 + a_1 x + \cdots + a_{n-1} x^{n-1} = \\
  (-1)^{1+j} m(1,j) +
  (-1)^{2+j} m(2,j) x + \cdots +
  (-1)^{k+j} m(k,j) x^{k-1} + \cdots + (-1)^{n+j} m(n,j) x^{n-1}
\end{multline}
From the above, we have an expression for the coefficients of the polynomial:
\begin{equation}
\label{eq:one_form_of_coeff}
 a_{k-1} = (-1)^{k+j} m(k,j)
\end{equation}

Note that if $x=y_k$ for $k\neq j$ then columns $k$ and $j$ are identical, the determinant is zero,
and therefore $p_{n-1}(x)$ has a root at $x=y_k$. Because $p_{n-1}(x)$ has degree $(n-1)$,
and all the $y$s are distinct, this is a complete list of the roots:
\begin{equation}
  x = y_1, \hspace{0.1in} x=y_2, \hspace{0.1in} \ldots, \hspace{0.1in} x=y_{j-1}, \hspace{0.1in} x=y_{j+1}, \hspace{0.1in} \ldots, \hspace{0.1in} x=y_n
\end{equation}
we immediately have the following alternative expression for $p_{n-1}(x)$
\begin{equation}
  p_{n-1}(x) = a_{n-1} (x-y_1) \cdots (x-y_{j-1}) (x-y_{j+1}) \cdots (x-y_n)
\end{equation}
removing all the parentheses we obtain yet another form of the same polynomial,
made from Vieta's formulas:
\begin{multline}
  p_{n-1}(x) = a_{n-1} x^{n-1} +
  a_{n-1} (-1)^{1} \sigma_{n-1, 1}(\widehat y_j) x^{n-2} + \cdots \\
  +  a_{n-1} (-1)^{n-k} \sigma_{n-1,n-k}(\widehat y_j) x^{k-1} + \cdots +
  a_{n-1} (-1)^{n-1} \sigma_{n-1,n-1}(\widehat y_j)
\end{multline}
So, we arrive at still another expression for the coefficients of the polynomial:
\begin{equation}
  a_{k-1} = a_{n-1} (-1)^{n-k} \sigma_{n-1,n-k}(\widehat y_j)
\end{equation}
Equating the two expressions for the coefficients, we immediately obtain a formula for $m(k,j)$
\begin{equation}
  m(k,j) = (-1)^{k+j} a_{k-1} =
    (-1)^{k+j} a_{n-1}(-1)^{n-k} \sigma_{n-1,n-k}(\widehat y_j) =
    (-1)^{n+j} a_{n-1} \sigma_{n-1,n-k}(\widehat y_j)
\end{equation}
since by \eqref{eq:one_form_of_coeff}, we know
$$a_{n-1} = (-1)^{n+j} m(n,j) = (-1)^{n+j} v(\widehat y_j)$$
we have therefore
\begin{equation}
  m(k,j) = (-1)^{n+j} a_{n-1} \sigma_{n-1,n-k}(\widehat y_j) =
  (-1)^{n+j} (-1)^{n+j} v(\widehat y_j) \sigma_{n-1,n-k}(\widehat y_j) =
  v(\widehat y_j) \sigma_{n-1,n-k}(\widehat y_j)
\end{equation}
as we claimed.
\end{proof}

\section{Proof of Lemma 2}\label{lemma_2_proof_appendix}

\begin{lemma_repeat}
  The sum $S$ has the following closed form:
  \begin{equation}
    S = \sum_{j=1}^n c_j \delta_j^n = (-1)^{k+n+1} k! \sigma_{n,n-k}(\delta)
  \end{equation}
\end{lemma_repeat}

\begin{proof}
Using the formula for $c_j$ in its not yet fully simplified form from \eqref{eq:formula_for_cj_raw}, we see that
\begin{equation}
  S = \sum_{j=1}^n c_j \delta_j^n
  = \sum_{j=1}^n \delta_j^n (-1)^{k+j+1} k! \frac{v(\widehat \delta_j) \sigma_{n-1, n-k-1}(\widehat \delta_j)} {v(\delta)}
  = \frac{k!}{v(\vec{\delta})} \sum_{j=1}^n \delta_j^n (-1)^{k+j+1} v(\widehat \delta_j) \sigma_{n-1, n-k-1}(\widehat \delta_j)
\end{equation}
at this point we observe that
\begin{equation}
  \delta_j \sigma_{n-1, n-k-1} (\widehat \delta_j) =
  \sigma_{n, n-k}(\vec{\delta}) - \sigma_{n-1,n-k}(\widehat\delta_j)
\end{equation}
as a general property of symmetric polynomials
so that we can continue with simplification of $S$:
\begin{eqnarray}
S & = &
\frac{k!}{v(\vec{\delta})} \sum_{j=1}^n \delta_j^{n-1} (-1)^{k+j+1} v(\widehat \delta_j) \left( \sigma_{n, n-k}(\vec{\delta}) - \sigma_{n-1,n-k}(\widehat\delta_j) \right) \\
& = &
\frac{k!}{v(\vec{\delta})} \sum_{j=1}^n \delta_j^{n-1} (-1)^{k+j+1} v(\widehat \delta_j) \sigma_{n, n-k}(\vec{\delta}) -
\frac{k!}{v(\vec{\delta})} \sum_{j=1}^n \delta_j^{n-1} (-1)^{k+j+1} v(\widehat \delta_j) \sigma_{n-1,n-k}(\widehat\delta_j). \label{long_equation}
\end{eqnarray}
Now, we observe that the first sum in \eqref{long_equation} is almost an expansion of a Vandermonde determinant along the last row, since
\begin{eqnarray}
  v(\delta) & = & \sum_{j=1}^n (-1)^{n+j}  \delta_j^{n-1} m(n,j) \\
  & = & \sum_{j=1}^n (-1)^{n+j} \delta_j^{n-1} v(\widehat \delta_j ) \sigma_{n-1, 0} (\widehat \delta_j ) \\
  & = & \sum_{j=1}^n (-1)^{n+j} \delta_j^{n-1} v(\widehat \delta_j) \\
\end{eqnarray}
which gives us
\begin{eqnarray}
    \label{eq:result_of_first_sum}
  \frac{k!}{v(\delta)} \sum_{j=1}^n \delta_j^{n-1} (-1)^{k+j+1} v(\widehat \delta_j) \sigma_{n, n-k}(\delta)
  & = & \sigma_{n-1,n-k}(\delta) \frac{1}{v(\delta)} k! (-1)^{k+1-n} \sum_{j=1}^n (-1)^{n+j} v(\widehat \delta_j) \delta_j^{n-1} \\
& = &  \sigma_{n-1, n-k}(\vec{\delta}) \frac{1}{v(\vec{\delta})} k! (-1)^{k+1-n} v(\vec{\delta}) \\
  & = & (-1)^{k+1-n} \sigma_{n-1, n-k} (\vec{\delta}) k!
  \end{eqnarray}
Having dispatched the first sum, we can now consider the second sum 
from \eqref{long_equation}; it involves the coefficients for $c_j$ for the $(k-1)$-th derivative. We use the original formula for $c_j$ for the $k$-th derivative
from \eqref{eq:formula_for_cj_raw} which is
\begin{equation}
  c_j^{(k)} = (-1)^{k+j+1} k! v(\widehat \delta_j) \sigma_{n-1,n-k-1} (\widehat\delta_j) \frac{1}{v(\delta)}
\end{equation}
to obtain
\begin{equation}
  c_j^{(k-1)} = (-1)^{k+j} (k-1)! v(\widehat \delta_j) \sigma_{n-1,n-k}(\widehat\delta_j) \frac{1}{v(\delta)}
\end{equation}
and observe that the second sum in \eqref{long_equation} can become
\begin{eqnarray} 
  \frac{k!}{v(\delta)} \sum_{j=1}^n \delta_j^{n-1} (-1)^{k+j+1} v(\widehat \delta_j) \sigma_{n-1,n-k}(\widehat\delta_j)
 & = &
  -k   \sum_{j=1}^n (k-1)! \delta_j^{n-1} (-1)^{k+j} v(\widehat \delta_j) \sigma_{n-1,n-k}(\widehat\delta_j) \frac{1}{v(\delta)} \\
  & = & -k \sum_{j=1}^n \delta_j^{n-1} c_j^{(k-1)} \\
  & = & -k \sum_{j=1}^n A_{n,j} c_j^{(k-1)} \\
  & = & -k b_n^{(k-1)} = 0
\end{eqnarray}
for all values of $k$ that satisfy $0 < k < n$. We used $b_n^{(k-1)}$ to denote the last coordinate of the right-hand side vector for the problem of finding the $(k-1)$-th derivative.

So, equation \eqref{eq:result_of_first_sum} contains all there is to the sum $S$, and the proof is complete.
\end{proof}

\section{Proof of Lemma 3}\label{lemma_3_proof_appendix}
We will need the following estimate on $S$ as well
\begin{lemma_repeat}
  The sum $S = \sum_{j=1}^n c_j \delta_j^n $ satisfies
  \begin{equation}
    |S| \le \sum_{j=1}^n | c_j \delta_j^n | \le \frac{ \Delta^{2n-k-1} n!  }{\epsilon^{n-1} (n-k-1)! }
\end{equation}
\end{lemma_repeat}

\begin{proof}
Using the formula \eqref{eq:formula_for_cj_raw} for $c_j$ and the fact that
\begin{equation}
\left| \frac{v(\check\delta_j)}{v(\delta)} \right| = \frac{1}{\prod_{i\not=j} | \delta_i - \delta_j |} \le \frac{1}{\epsilon_j^{n-1}}
\end{equation}
we see that
\begin{multline}
  |c_j \delta_j^n | = k! \left| \frac{ v(\check \delta_j) }{v(\delta)} \right| |\sigma_{n-1, n-k-1}(\check\delta_j)| |\delta_j^n | \le
  k! \frac{1}{\epsilon_j^{n-1}} \binom{n-1}{n-k-1} \Delta^{n-k-1} \Delta^{n} \le  \frac{ k! \Delta^{2n-k-1} }{\epsilon^{n-1}}  \binom{n-1}{k} \\
  = \frac{ k! \Delta^{2n-k-1} }{\epsilon^{n-1}}  \binom{n-1}{k}
  = \frac{ \Delta^{2n-k-1} (n-1)!  }{\epsilon^{n-1} (n-k-1)! }
\end{multline}
and
\begin{equation}
  \sum_{j=1}^n |c_j \delta_j^n | \le
  \sum_{j=1}^n   \frac{ \Delta^{2n-k-1} (n-1)!  }{\epsilon^{n-1} (n-k-1)! } =  \frac{ \Delta^{2n-k-1} n!  }{\epsilon^{n-1} (n-k-1)! }
\end{equation}
\end{proof}

\section{When the Algorithm will Under/Over Estimate}
\label{appendix:under_over_estimates}
Sometimes it is useful to know if the algorithm's formula will always underestimate or always overestimate
the value of $f^{(k)}(t)$. Here, we show that
Theorem \ref{theorem:little_o_error_term} gives us the tools to prove the following additional result.
\setcounter{theorem_repeat}{2}
\begin{theorem_repeat}
  Assume that all $\delta_j < 0$
  \begin{enumerate}
    \item If $f^{(n)}(t) > 0$, then the error term $R$ is positive for all sufficiently small values of $h$, and the estimate of $f^{(k)}(t)$ provided by \eqref{eq:final_result_for_algorithm} is an underestimate.
    \item If $f^{(n)}(t) < 0$, then the error term $R$ is negative for all sufficiently small values of $h$ and the estimate of $f^{(k)}(t)$ provided by \eqref{eq:final_result_for_algorithm} is an overestimate. 
    \end{enumerate}
\end{theorem_repeat}

\begin{proof}
  Note that if $n-k$ is even, then $(-1)^{n-k} = 1$ and $\sigma_{n,n-k}(\vec{\delta}) > 0$ as a sum of products of even number of negative $\delta_i$'s.
  Also, if $n-k$ is odd, then $(-1)^{n-k} = -1$ and $\sigma_{n,n-k}(\vec{\delta}) < 0$ as a sum of products of odd number of negative $\delta_i$'s.
  In either case,
  \begin{equation}
    \label{eq:always_positive}
    (-1)^{n-k} \sigma_{n,n-k}(\vec{\delta}) > 0
  \end{equation}

  So, if $f^{(n)}(t)>0$, then
  \[ \frac{(-1)^{n-k} h^{n-k} k!  f^{(n)}(t) \sigma_{n,n-k}(\vec{\delta}) }{n!} > 0 \]
  and according to \eqref{eq:R_in_little_o_form}, the error term $R$ is positive for all small enough values of $h$.
Similarly, if $f^{(n)}(t)<0$, then
  \[ \frac{(-1)^{n-k} h^{n-k} k!  f^{(n)}(t) \sigma_{n,n-k}(\vec{\delta}) }{n!} < 0 \]
  and according to \eqref{eq:R_in_little_o_form}, the error term $R$ is negative for all small enough values of $h$.
\end{proof}

\section{Explicit Form of Coefficients in Special Cases}
\label{appendix_special_cases}
We derive some explicit formulas for coefficients produced by algorithm \eqref{eq:formula_for_cj_raw} in a couple of special cases that are common in applications.

Assume $\delta_i = -i$ for $i=1, \ldots, n$, the formula for coefficients $c_j$ simplifies in the case $k=1$ to
\begin{equation}
    \label{eq:special_case_k_is_one}
  c_j = (-1)^j \binom{n}{j} \left( - \frac{1}{j} + 1 + \frac{1}{2} + \cdots + \frac{1}{n}\right),
\end{equation}
and in the case $k=2$ to
\begin{equation}
    \label{eq:special_case_k_is_two}
  c_j = (-1)^j 2 \binom{n}{j} \left( \sigma_{n,2}\left(1, \frac12, \ldots, \frac1n\right) - \frac1j \left(- \frac{1}{j} + 1 + \frac{1}{2} + \cdots + \frac{1}{n}\right) \right)
\end{equation}

\emph{Proof.} We start with the general formula
\begin{equation}
\label{eq:formula_for_cks}
c_j = \frac{(-1)^{n-k} k!\sigma_{n-1,n-k-1}(\check \delta_j )}
{\prod_{i\not=j} \left ( \delta_j - \delta_i \right ) }
\end{equation}
and observe that for $\delta_i = -i$ the product in the denominator is
\begin{multline}
\label{eq:denominator_for_cks}
\prod_{i\not=j} \left ( \delta_j - \delta_i \right ) =
\prod_{i\not=j} \left ( -j + i \right ) \\ =
(-j + 1) (-j + 2) \cdots (-1) \cdot (1) \cdots (-j+n) \\ =
(-1)^{j-1} (j-1)! (n-j)!
\end{multline}

For $k=1$, the symmetric polynomial $\sigma_{n-1,n-k-1}$ becomes
\begin{equation}
\sigma_{n-1,n-k-1}(\check \delta_j) =
\sigma_{n-1,n-2}(\check \delta_j) =
\frac{1}{\delta_j} \delta_1 \ldots \delta_n \left( -\frac{1}{\delta_j} + \sum_{i=1}^n \frac{1}{\delta_i}
\right)
\end{equation}
which with $\delta_i = -i$ transforms into
\begin{equation}
\label{eq:symmetric_polynomial_for_k_one}
\sigma_{n-1,n-k-1}(\check \delta_j) =
-\frac{1}{j} (-1)^n n! \left( \frac{1}{j} - \sum_{i=1}^n \frac{1}{i} \right) =
(-1)^n \frac{1}{j} n! \left( -\frac{1}{j} + \sum_{i=1}^n \frac{1}{i} \right)
\end{equation}
Putting the results from \eqref{eq:denominator_for_cks},  and \eqref{eq:symmetric_polynomial_for_k_one} into \eqref{eq:formula_for_cks} for $k=1$, we arrive at
\begin{multline}
c_j = \frac{ (-1)^{n-1}  (-1)^n \frac{1}{j} n! \left( -\frac{1}{j} + \sum_{i=1}^n \frac{1}{i} \right) }{ (-1)^{j-1} (j-1)!(n-j)! }
= (-1)^j \frac{n!}{j! (n-j)!} \left( -\frac{1}{j} + \sum_{i=1}^n \frac{1}{i} \right) \\
= (-1)^j \binom{n}{j} \left( -\frac{1}{j} + \sum_{i=1}^n \frac{1}{i} \right)
\end{multline}
which is exactly formula \eqref{eq:special_case_k_is_one}.

For $k=2$, the symmetric polynomial $\sigma_{n-1,n-k-1}$ becomes
\begin{equation}
\label{eq_symmetric_polynomial_for_k_two_start}
\sigma_{n-1,n-3}(\check \delta_j) =
\frac{1}{\delta_j} \delta_1 \cdots \delta_n \cdot \sigma_{n-1,2}\left(\frac{1}{\delta_1}, \ldots, \frac{1}{\delta_{j-1}}, \frac{1}{\delta_{j+1}}, \ldots, \frac{1}{\delta_n} \right)
\end{equation}
focusing on the $\sigma_{n-1,2}\left(\frac{1}{\delta_1}, \ldots, \frac{1}{\delta_{j-1}}, \frac{1}{\delta_{j+1}}, \ldots, \frac{1}{\delta_n} \right)$ term, we see that
\begin{multline}
  \sigma_{n-1,2}\left(\frac{1}{\delta_1}, \ldots, \frac{1}{\delta_{j-1}}, \frac{1}{\delta_{j+1}}, \ldots, \frac{1}{\delta_n} \right) \\ =
  \sigma_{n,2}\left( \frac{1}{\delta_1}, \ldots, \frac{1}{\delta_n} \right) - \frac{1}{\delta_j} \left( \frac{1}{\delta_1} + \cdots + \frac{1}{\delta_n} - \frac{1}{\delta_j} \right)
\end{multline}
which for $\delta_i = -i$ becomes
\begin{multline}
    \label{eq_symmetric_polynomial_for_k_two_end}
  \sigma_{n-1,2}\left(\frac{1}{\delta_1}, \ldots, \frac{1}{\delta_{j-1}}, \frac{1}{\delta_{j+1}}, \ldots, \frac{1}{\delta_n} \right) \\ =
  \sigma_{n-1,2}\left(1, \ldots, \frac{1}{j-1}, \frac{1}{j+1}, \ldots, \frac{1}{n} \right) \\ =
  \sigma_{n,2} \left(1, \frac12, \ldots, \frac1n\right) - \frac{1}{j} \left( 1 + \cdots + \frac{1}{n} - \frac{1}{j} \right)
\end{multline}
So putting the results of \eqref{eq:denominator_for_cks}, \eqref{eq_symmetric_polynomial_for_k_two_start} and \eqref{eq_symmetric_polynomial_for_k_two_end} into \eqref{eq:formula_for_cks} for $k=2$ we obtain
\begin{multline}
c_j = \frac{(-1)^{n-2} \cdot 2!  \frac{1}{-j} (-1)^n n!  \left[ \sigma_{n,2} \left(1, \frac12, \ldots, \frac1n\right) - \frac{1}{j} \left( 1 + \cdots + \frac{1}{n} - \frac{1}{j} \right) \right]}{(-1)^{j-1} (j-1)!(n-j)!}
 \\ =
(-1)^j\cdot 2 \cdot \frac{n!}{j!(n-j)!} \left[ \sigma_{n,2} \left(1, \frac12, \ldots, \frac1n\right) - \frac{1}{j} \left( 1 + \cdots + \frac{1}{n} - \frac{1}{j} \right) \right] \\ =
(-1)^j \cdot 2 \cdot \binom{n}{j} \left[ \sigma_{n,2} \left(1, \frac12, \ldots, \frac1n\right) - \frac{1}{j} \left( 1 + \cdots + \frac{1}{n} - \frac{1}{j} \right) \right]
\end{multline}
which is exactly formula \eqref{eq:special_case_k_is_two}.

\end{document}